\documentclass[reqno]{amsart}

\usepackage{bbm}

\usepackage[a4paper,left=3.5cm,right=3.5cm,top=2cm]{geometry}
\usepackage{amsfonts}
\usepackage{amsthm}
\usepackage{amscd}
\usepackage{amsmath}
\usepackage{txfonts}
\usepackage{mathrsfs}
\usepackage{extarrows}
\usepackage[colorlinks, linkcolor=blue]{hyperref}

\theoremstyle{plain}
\newtheorem{thm}{Theorem}[section]
\newtheorem{prop}[thm]{Proposition}
\newtheorem{lem}[thm]{Lemma}
\newtheorem{cor}[thm]{Corollary}

\newtheorem{defn}{Definition}[section]

\theoremstyle{remark}
\newtheorem*{clm}{Claim}
\newtheorem*{rmk}{Remark}

\newcommand{\sgn}{\mathrm{sgn}\;\!}
\newcommand{\ve}{\varepsilon}

\newcommand{\mbbr}{\mathbb{R}}

\newcommand{\mclc}{\mathcal{C}}

\newcommand{\mclh}{\mathcal{H}}

\newcommand{\mcll}{\mathcal{L}}
\newcommand{\mclm}{\mathcal{M}}

\newcommand{\msca}{\mathscr{A}}

\newcommand{\mscc}{\mathscr{C}}

\newcommand{\mscp}{\mathscr{P}}
\newcommand{\mscq}{\mathscr{Q}}

\newcommand{\msct}{\mathscr{T}}

\newcommand{\hc}{\hat{c}}

\newcommand{\tc}{\tilde{c}}

\newcommand{\hI}{\hat{I}}

\newcommand{\dif}{\;\!\mathrm{d}\:\!}

\newcommand{\bma}{\mathbbm{a}}

\newcommand{\dom}{\mathrm{Dom}}

\newcommand{\cl}{\mathrm{cl}\,}

\newcommand{\inr}{\mathrm{int}\,}

\newcommand{\slowrec}{\msca}

\newcommand{\typical}{\msct}

\begin{document}

\title{On fair entropy of the tent family}

\author[B. Gao]{Bing Gao\mbox{\ \ }}
\address{Bing Gao: School of Mathematics, Hunan University, Changsha 410082, China}
\email{binggao@hnu.edu.cn}

\author[R. Gao]{\mbox{\ \ }Rui Gao}
\address{Rui Gao: College of Mathematics, Sichuan University, Chengdu 610064, China}
\email{gaoruimath@scu.edu.cn}

\maketitle

\begin{abstract}

  The notions of fair measure and fair entropy were introduced by Misiurewicz and Rodrigues \cite{MR18} recently, and discussed in detail for piecewise monotone interval maps. In particular, they showed that the fair entropy $h(a)$ of the tent map $f_a$, as a function of the parameter $a=\exp(h_{top}(f_a))$, is continuous and strictly increasing on $[\sqrt{2},2]$. In this short note, we extend the last result and characterize regularity of the function $h$ precisely. We prove that $h$ is $\frac{1}{2}$-H\"{o}lder continuous on $[\sqrt{2},2]$ and identify its best H\"{o}lder exponent on each subinterval of $[\sqrt{2},2]$. On the other hand,  parallel to a recent result on topological entropy of the quadratic family due to  Dobbs and Mihalache \cite{DM19}, we give a formula of pointwise  H\"{o}lder  exponents of $h$ at parameters chosen in an explicitly constructed  set of full  measure. This formula particularly implies that the derivative of $h$ vanishes almost everywhere.

\end{abstract}

\section{Introduction}\label{se:intro}

%\subsection{Tent family}

Consider the family $f_a:I_a\to I_a$ of tent maps, where $I_a=[1-a,1]$, $f_a(x)=1-a|x|$ with critical (or turning) point $c=0$, $a\in[\sqrt{2},2]$. Denote the critical orbit of $f_a$ by $c_n(a):=f_a^n(c)$. Denote $f_a^{-1}c_3(a)=\{c_2(a),\hc_2(a)\}$, i.e. $\hc_2(a)=-c_2(a)=a-1$. Sometimes we drop the dependence of $c_n$ on $a$ for short.

%\iffalse

The notions of fair measure and fair entropy were introduced by Misiurewicz and Rodrigues \cite{MR18} recently for finite-to-one surjective continuous maps. The motivation and intuitive meaning of these notions are clearly elaborated in their paper. For tent maps, these notions read as follows. Given %\footnote{In \cite{MR18} the authors focus on $a\in(\sqrt{2},2]$ indeed, because $f_a:I_a\to I_a$ is topologically mixing iff $a\in(\sqrt{2},2]$. The reason we add the end point $\sqrt{2}$ here is insignificant and merely to make the range of $a$ compact.}
$a\in[\sqrt{2},2]$, there exists an atomless $f_a$-invariant Borel probability $\mu_a$, called the {\bf fair measure} of $f_a$, which can be characterized by one of the following two equivalent conditions:
\begin{itemize}
  \item $\mu_a$ is the unique conformal measure of $f_a$ with respect to (w.r.t. for short) the Jacobian $j_a$ defined below
  \[
  j_a(x):=\#f_a^{-1}(f_a x) =\left\{
  \begin{array}{ccl}
    2 &,&   x\in [c_2,\hc_2]\setminus\{c\} \\
    1 &,&   x\in \{c\}\cup (\hc_2, c_1]
  \end{array}
  \right.,
  \]
  in the sense that for any Borel set $E\subset I_a$,

  \begin{equation}\label{eq:fair conformal}
    \text{ $f_a$ is injective on $E$ } \implies   \mu_a (f_a E)=\int_E j_a \dif \mu_a.
  \end{equation}
  % $\mu_a (f_a E)=\int_E \frac{1}{w_a} \dif \mu_a$ for Borel set $E$ on which $f_a$ is injective.

  \item $\mu_a$ is the unique equilibrium state of  $f_a$ w.r.t. the potential $-\log j_a$ (and pressure $0$) in the sense that
  \begin{equation}\label{eq:fair var}
    0=  h_{\mu_a}(f_a) -\int \log j_a\dif \mu_a  \ge h_\nu(f_a) -\int \log j_a\dif \nu  \,,
  \end{equation}
where $\nu$ is any $f_a$-invariant Borel probability, and ``unique" means that the equality in ``$\ge$" holds only if $\nu=\mu_a$. \eqref{eq:fair var} is a special case of variational principle, whose validity can be guaranteed by \cite[Theorem~3.3]{Ba00}, for example; the uniqueness part follows from \cite[Proposition~3.5]{Ba00}.
  \end{itemize}
The measure-theoretic entropy $h_{\mu_a}(f_a)$ is called the {\bf fair entropy} of $f_a$.  Define
\begin{equation}\label{eq:fair function}
  \mclh(a):=\frac{h_{\mu_a}(f_a)}{2\log 2} ,\quad a\in [\sqrt{2},2]
\end{equation}
for convenience. It is shown in \cite[Theorem~5.13]{MR18}  that $\mclh:[\sqrt{2},2]\to [\frac{1}{4},\frac{1}{2}]$ is an increasing homeomorphism. In this paper we are concerned about the regularity of $\mclh$.
%\textcolor{red}{in order to get better understanding of the concept of fair entropy from the simple model of tent maps, especially how the fair entropy varies in a typical one-parameter family}.
Our first main result is the following.

\begin{thm}\label{thm:interval}
There exist an increasing sequence (specified in Proposition~\ref{prop:window}) of parameters $\sqrt{2}=\bma_2<\bma_3<\cdots $ approaching to $2$ , such that the following hold for each $r\ge 2$.
\begin{enumerate}
  \item  Given $\bma_r <b\le 2$, $\mclh$ is H\"{o}lder continuous on $[\bma_r,b]$ with H\"{o}lder exponent $\alpha(r,b)$ defined below:
  \begin{equation}\label{eq:holder exp}
     \alpha(r,b):=\frac{(r-1)\log 2}{r\log b} \,.
  \end{equation}
  In particular, $\mclh$ is $\frac{1}{2}$-H\"{o}lder on $[\sqrt{2},2]$.
  \item The H\"{o}lder exponents given by \eqref{eq:holder exp}  are optimal in the following sense: for any subinterval $J$ of $[\sqrt{2},2]$ with $b=\sup J\le \bma_{r+1}$ and any $\alpha>\alpha(r,b)$, $\mclh$ is not $\alpha$-H\"{o}lder continuous on $J$.
\end{enumerate}

\end{thm}

 Our second main result is Theorem~\ref{thm:pointwise} below, motivated by a parallel (but much deeper) result of Dobbs and Mihalache \cite{DM19} recently on topological entropy of the quadratic family. Let us introduce some notations first. Define the {\bf pointwise H\"{o}lder exponent} of $\mclh$ as follows, provided that the limit on right hand side (RHS for short) exists:
\begin{equation}\label{eq:point holder exp def}
  \beta(a):=\lim_{b\to a}\frac{\log|\mclh(b)-\mclh(a)|}{\log|b-a|}.
\end{equation}
For each $n\ge 1$, denote
 \begin{equation}\label{eq:return count}
    \Gamma_n(a):=\#\{1\le k\le n : c_k(a) <\hc_2(a)\}.
 \end{equation}
$\beta$ is closely related to the quantity $\gamma$ defined below, provided that the limit on RHS exists:
\begin{equation}\label{eq:point holder exp def}
  \gamma(a):=\lim_{n\to\infty}\frac{1}{n}\Gamma_n(a)\in [0,1] .
\end{equation}
%By definition, $\gamma(a)\le 1$.

%Denote
%\[
%\mscp=\{a\in [\sqrt{2},2]: \exists k \text{ s.t } c_k(a)=c \text{ or } c_k(a)=\hc_2(a)\}.
%\]

\begin{thm}\label{thm:pointwise}
  There exists a Borel set  $\slowrec\subset [\sqrt{2},2]$ (specified in Definition~\ref{def:slow rec}) of full Lebesgue measure in  $[\sqrt{2},2]$  satisfying the properties below.
  \begin{enumerate}
    \item If $a\in \slowrec$, then $\beta(a)$ is well-defined iff $\gamma(a)$ is well-defined, and they are related by

  \begin{equation}\label{eq:pointwise holder}
    \beta(a)=\frac{\log 2}{\log a}\cdot \gamma(a)\,.
  \end{equation}
  \item Let $\dom(\gamma)$ denote the collection of $a$ such that $\gamma(a)$ is well-defined. The following hold.
  \begin{itemize}
    \item $\slowrec\cap \{a\in\dom(\gamma) : \gamma(a)=1\}$ is dense in $[\sqrt{2},2]$.
    \item For each $r\ge 2$, $\gamma\ge \frac{r-1}{r}$ on $\dom(\gamma)\cap (\bma_r,\bma_{r+1})$, and $\slowrec\cap \{a\in\dom(\gamma) : \gamma(a)=\frac{r-1}{r}\}$ is dense in $(\bma_r,\bma_{r+1})$.
  \end{itemize}

  \end{enumerate}

\end{thm}

 It is well known that $f_a:I_a\to I_a$,  $a\in[\sqrt{2},2]$ admits a unique a.c.i.p. $\nu_a$ supported on $I_a$, which is also the unique measure of maximal entropy, i.e. $h_{\nu_a}(f_a)=\log a =h_{top}(f_a)$. Following Bruin \cite{Br98}, a parameter $a$ is called {\bf typical}, if for any bounded and Lebesgue-a.e.  continuous test function $g:I_a\to \mbbr$, the time average of $g$ along critical orbit of $f_a$ exists and coincides with its phase average w.r.t. $\nu_a$, i.e.
\begin{equation}\label{eq:typical}
  \lim_{n\to\infty}\frac{1}{n}\sum_{k=0}^{n-1} g(c_k(a)) =\int_{I_a}g\dif\nu_a \,.
\end{equation}
Denote the collection of typical parameters by $\typical$. According to Bruin~\cite{Br98}, $\typical$ is of full Lebesgue measure in $[\sqrt{2},2]$.

\begin{cor}
  If $a\in\slowrec\cap\typical$\,, then $\beta(a)>1$ is well-defined and $\mclh'(a)=0$ consequently. In particular, $\mclh'(a)=0$ for Lebesgue almost every $a\in [\sqrt{2},2]$, and hence $\mclh$ is not absolutely continuous on any subinterval of $[\sqrt{2},2]$.
\end{cor}

\begin{proof}
   Given $a\in\slowrec\cap\typical$, substituting $g=1_{[c_2,\hc_2)}$ into \eqref{eq:typical} yields that $\gamma(a)=\lim\limits_{n\to\infty}\frac{1}{n}\Gamma_n(a)=\nu_a([c_2,\hc_2])$ exists. Since $\nu_a\ne \mu_a$, substituting $\nu=\nu_a$  into the variational principle \eqref{eq:fair var}, the uniqueness of equilibrium state implies that $\nu_a([c_2,\hc_2])>\frac{\log a}{\log 2}$.  Then by the first assertion in Theorem~\ref{thm:pointwise}, $\beta(a)>1$ is well-defined. The rest are clear.
\end{proof}

As direct application of our main results, we can say something more about dynamical or geometric properties of the fair measure $\mu_a$ and/or associated fair entropy $h_{\mu_a}(f_a)$. The following is such an example. Basic knowledge in dimension theory asserts that $\mu_a$ is exact dimensional with Hausdorff dimension $\dim_H(\mu_a)=\frac{h_{\mu_a}(f_a)}{\log a}$, i.e. $\lim\limits_{\delta\to 0^+}\frac{\log \mu_a \big((x-\delta,x+\delta)\big)}{\log \delta}=\frac{h_{\mu_a}(f_a)}{\log a}$ for $\mu_a$-a.e. $x\in I_a$. Then  Theorem~\ref{thm:interval} is still valid for the function $a\mapsto \dim_H(\mu_a)$ instead of $\mclh$. On the other hand, as an immediate corollary of Theorem~\ref{thm:pointwise}, we have:

\begin{cor}
  The function $a\mapsto \dim_H(\mu_a)$ is not monotone on any subinterval of $[\sqrt{2},2]$.
\end{cor}

The paper is organized as follows. In \S~\ref{se:fair}  we discuss fair measure and fair entropy for unimodal maps and apply to tent maps. In \S~\ref{se:holder}, we focus on H\"{o}lder continuity of $\mclh$ on intervals: the first assertion in Theorem~\ref{thm:interval} is proved in \S~\ref{sse:interval proof}; as a by-product, H\"{o}lder continuity of individual distribution function of fair measure is proved in \S~\ref{sse:distribution holder}.  In \S~\ref{se:pointwise}, we mainly discuss pointwise H\"{o}lder exponents of $\mclh$: the first assertion in Theorem~\ref{thm:pointwise} is proved in \S~\ref{sse:pointwise proof}; the second assertion in Theorem~\ref{thm:pointwise} is proved in \S~\ref{sse:extremal} and the second assertion in Theorem~\ref{thm:interval} follows as a corollary.

\section{Fair measure and fair entropy of unimodal maps}\label{se:fair}

In \cite[\S~5]{MR18}, it is proved that for piecewise monotone interval map $f:I\to I$, if $f$ is surjective and topologically mixing, then it admits a unique fair measure. We restrict our discussion to unimodal case and elaborate more details on this result below.

\subsection{General discussion}\label{sse:general}
Let $I=[c_2,c_1]$ be a closed interval and let $f:I\to I$ be continuous. By saying that $f:I\to I$ is a {\bf mixing unimodal map}, we mean:
\begin{itemize}
  \item There exists $c$ in the interior of $I$, which is the unique critical (or turning) point of $f$,  such that $f$ is strictly increasing on $[c_2,c]$ and strictly decreasing on $[c,c_1]$.
  \item $f(c)=c_1$ and $f(c_1)=c_2$. In particular, $f$ is surjective on $I$.
  \item $f:I\to I$ is topologically mixing.
\end{itemize}

Given a mixing unimodal map $f:I\to I$, denote $c_n:=f^n(c)$ for $n\ge 0$ and denote $f^{-1}c_3=\{c_2,\hc_2\}$. An atomless\footnote{The atomless assumption is not included in the definition of fair measure in \cite{MR18}, but it turns out that following the definition in \cite{MR18},  fair measures we are discussing here are always atomless. Therefore we may merge the atomless assumption into definition for convenience.} Borel probability measure $\mu$ on $I$ is called a {\bf fair measure} of $f$, if the following holds for any Borel set $E\subset I$:
\[
\mu(E) = \frac{1}{2} \cdot \mu\Big(f\big(E\cap [c_2,c]\big) \Big) +   \frac{1}{2} \cdot \mu\Big(f\big(E\cap [c,\hc_2]\big) \Big) +  \mu\Big(f\big(E\cap [\hc_2,c_1]\big) \Big).
\]
In other words, $\mu$ is an atomless conformal measure of $f$ with respect to Jacobian $j_f$ in the following sense (as mentioned for tent maps at the beginning of \S~\ref{se:intro}):

\[
\text{$f$ is injective on $E$} \implies \mu(fE)=\int_E j_f\dif \mu\,, \quad j_f(x):= \#f^{-1}(fx)= \left\{ 
\begin{array}{ccl}
    2 &,&   x\in [c_2,\hc_2]\setminus\{c\} \\
    1 &,&   x\in \{c\}\cup (\hc_2, c_1]
  \end{array}\right..
\]
In \cite[\S~5]{MR18} it is proved that for any mixing unimodal map $f:I\to I$, it admits a unique atomless fair measure $\mu$ supported on $I$, and $\mu$ is automatically $f$-invariant. The measure-theoretic entropy $h_\mu(f)$ of the fair measure $\mu$ is called the {\bf fair entropy} of $f$, which is an invariant under topological conjugacy.  It is easy to see that
\[
h_\mu(f) = \int_I j_f \dif\mu=\log 2 \cdot \mu \big([c_2,\hc_2]\big) = 2\log 2\cdot \mu\big([c_2,c]\big).
\]

Let us restate the definition of fair measure in terms of its distribution. To this end, let\footnote{In \cite{MR18}, the same operator $\Phi$ in essence is acted on space of measures rather than (signed) distributions. }
\[
\Phi: \{F:[c_2,c_1]\to\mbbr\mid F(c_2)=0, F\text{ is continuous and of bounded variation}\}\circlearrowleft
\]
be the linear operator defined by:
\begin{equation}\label{eq:transfer}
  \Phi F(x)=
\left\{
\begin{array}{ccl}
\frac{1}{2}\big(F(fx)-F(c_3)\big)          &,&  c_2\le x \le c    \\
F(c_1)-\frac{1}{2}\big(F(fx)+F(c_3)\big)   &,&  c \le x \le \hc_2 \\
F(c_1)-F(f x)  &,&  \hc_2 \le x \le c_1 \\
\end{array}
\right..
\end{equation}
Then the distribution function $x\mapsto \mu \big([c_2, x]\big)$, $x\in [c_2,c_1]$ of the fair measure $\mu$ is nothing but the unique fixed point $F$ of
$\Phi$ that satisfies $F(c_1)=1$. From now on let $F$ denote this distribution.   Then the relation $F=\Phi F$ can be expressed as:
\begin{equation}\label{eq:distribution recurrence}
  F = \xi \cdot (F \circ f -1) + H\cdot 1_{[c_2,\hc_2)}\,,
\end{equation}
where\footnote{Actually we do not care how to evaluate $\xi$ at its discontinuities $c$ and $\hc_2$; the only restriction is to make \eqref{eq:distribution recurrence} valid. }
\[
\xi(x)=
\left\{
\begin{array}{ccl}
\frac{1}{2}         &,&  c_2\le x \le c    \\
-\frac{1}{2} &,&  c < x < \hc_2 \\
 -1 &,&  \hc_2 \le x \le c_1 \\
\end{array}
\right.,
\]
and
\[
H=\frac{1-F(c_3)}{2}=F(c)=\frac{h_\mu(f)}{2\log 2}.
\]
Denote
\[
  \xi|_m^n:=\prod_{k=m}^{n-1}\xi\circ f^k
 \ \ \text{and}\ \  \xi|_m^m:=1, ~~~~0\le m<n\,.
\]
Iterating \eqref{eq:distribution recurrence} we obtain that:
\begin{equation}\label{eq:fair distribution iteration}
  F=  \xi|_0^n \cdot F\circ f^n - \sum_{k=1}^{n} \xi|_0^k + H\cdot \sum_{k=0}^{n-1} \xi|_0^k\cdot1_{[c_2,\hc_2)} \circ f^k  \,.
\end{equation}
Evaluating \eqref{eq:fair distribution iteration} at $x=c$, we have:
\[
H=F(c) = \xi|_0^n(c) \cdot F(c_n) - \sum_{k=1}^{n} \xi|_0^k(c) + H\cdot \sum_{k=0}^{n-1} \xi|_0^k(c)\cdot1_{[c_2,\hc_2)}(c_k) \,.
\]
It can be rewritten as:
\begin{equation}\label{eq:fair entropy expansion}
  - \sum_{k=1}^{n-1} \xi|_1^k(c)\cdot1_{[c_2,\hc_2)}(c_k) \cdot H =  \xi|_1^n(c) \cdot F(c_n) - \sum_{k=1}^{n} \xi|_1^k(c)\,.
\end{equation}

\begin{rmk}
From $f:I\to I$ being topologically mixing it is easy to see that $f([\hc_2,c_1])=[c_2,c_3]\subset [c_2,\hc_2)$. It follows that $|\:\!\xi|_1^{2n+1}(c)\:\!|\le 2^{-n}$ for each $n\ge 1$ and $0\ne \sum_{n=1}^{\infty} \xi|_1^n(c)\cdot 1_{[c_2,\hc_2)}(c_{n})\in (-2,-1)$. Then letting $n\to\infty$ in \eqref{eq:fair entropy expansion}, we obtain an explicit expression of the fair entropy below:
\begin{equation}\label{eq:fair series}
      \frac{h_\mu(f)}{2\log 2} = H = \frac{\sum_{n=1}^{\infty} \xi|_1^n(c)} {\sum_{n=1}^{\infty} \xi|_1^n(c)\cdot 1_{[c_2,\hc_2)}(c_{n}) }.
\end{equation}
The form of \eqref{eq:fair series} might suggest one considering to use a weighted version of the Milnor-Thurston kneading theory in \cite{MT} to study the fair entropy. Such a theory was developed by Baladi and Ruelle \cite{BR94}, and by Rugh and Tan \cite{RT15}. Unfortunately, we cannot see how to use this theory to simplify proofs at this moment. We take \eqref{eq:fair entropy expansion} rather than \eqref{eq:fair series} as the starting point of our argument.
\end{rmk}

\subsection{Applying to tent maps}

In this subsection we apply the discussion in \S~\ref{sse:general} to tent maps $f_a:I_a\to I_a$, $a\in(\sqrt{2},2]$, which are mixing unimodal maps. Recall that $c_n=c_n(a)=f_a^n(c)$. In particular, $c_0=c=0$, $c_1=1$ and $c_2(a)=1-a=-\hc_2(a)$.  Denote

\[
\ve_a(x)= -\,\sgn\,x =
\left\{\begin{array}{ccc}
1        &,&  c_2(a) \le x <c    \\
-1 &,&  c  < x \le c_1
\end{array}
\right.,
\quad \xi_a=\frac{\ve_a}{ j_a} =
\left\{
\begin{array}{ccc}
\frac{1}{2}         &,&  c_2(a) \le x <c    \\
-\frac{1}{2}   &,&  c < x < \hc_2(a) \\
-1 &,&  \hc_2(a) < x < c_1 \\
\end{array}
\right..
\]
In particular, for $\xi_a$ we consider its domain as
\[
\dom (\xi_a) =I_a\setminus\{c,\hc_2(a)\} = [c_2(a),c) \sqcup (c,\hc_2(a))\sqcup (\hc_2(a),c_1].
\]
Note that on $\dom (\xi_a)$, $\xi_a$ coincides with $\xi$ introduced in \S~\ref{sse:general} for $f=f_a$.

Following \cite{CKY} and \cite{BM96}, for each $n\ge 0$, the map $a\mapsto c_n(a)$ is also denoted by $\varphi_n$, especially when we want to emphasize how $c_n(a)$ changes as $a$ varies. %The name ``lap" introduced below is borrowed from \cite{BM96} but has different meaning here.

\begin{defn}
  Given $n\ge 3$, a connected component of
  \[
  \{ a\in (\sqrt{2},2): \varphi_k(a)\in \dom(\xi_a), 3\le k< n\} = \{ a\in (\sqrt{2},2): \varphi_k(a)\ne c \text{ or } \hc_2(a),  3\le k< n\}
  \]
  is called a {\bf lap} of $\varphi_n$ (w.r.t. $\xi_a$). In particular, $(\sqrt{2},2)$ is the only lap of $\varphi_3$.
\end{defn}

Now let us apply the analysis in \S~\ref{sse:general} to tent maps. Recall that  the unique fair measure of $f_a$ is denoted by $\mu_a$,  and let  $F_a$  denote the distribution of $\mu_a$ from now on. Substituting $f=f_a$, $\xi=\xi_a$ (on $\dom(\xi_a)$) and $F=F_a$ into \eqref{eq:fair entropy expansion}, we obtain the following.

\begin{lem}\label{lem:entropy difference}
  Let $J\subset (\sqrt{2},2)$ be a lap of $\varphi_{n}$ for some $n\ge 3$. Then
 \[
 A_{n,J}:=\sum_{k=1}^{n-1} \xi_a|_1^k(c)\cdot 1_{[c_2,\hc_2)}(c_k(a))\,, \quad B_{n,J}:=\xi_a|_1^n(c)=\prod_{k=1}^{n-1}\xi_a(c_k(a))
 \]
 are constant on $J$,  and
  \begin{equation}\label{eq:fair distribution differ}
 - A_{n,J}\cdot\Big(\mclh(a)-\mclh(a')\Big)=B_{n,J} \cdot\Big( F_a(c_n(a)) - F_{a'}(c_n(a')) \Big)\,,\ \  \forall~  a,a'\in J.
  \end{equation}
Moreover,
\begin{itemize}
  \item if $n=3$, then  $A_{n,J}=-1$, $B_{n,J}=-\frac{1}{2}$;
  \item if $n\ge 4$, then  $-2<A_{n,J}<-1$ and $|B_{n,J}|= 2^{-\Gamma_{n-1}(a)}\le \frac{1}{4}$ for any $a\in J$.
\end{itemize}
\end{lem}
\begin{proof}
  Since $J$ is a lap of $\varphi_{n}$, when $1\le k<n$, $c_k(a)\notin\{c,\hc_2(a)\}$, so $\xi_a(c_k(a))$ and $1_{[c_2,\hc_2)}(c_k(a))$ are constant on $J$ by continuity. It follows that $A_{n,J}$ and $B_{n,J}$ are constant on $J$. Substituting $f=f_a$  and $f=f_{a'}$ into \eqref{eq:fair entropy expansion} respectively and taking their difference, \eqref{eq:fair distribution differ} follows. The statements on the values of $A_{n,J}$ and $B_{n,J}$ follow from direct calculation and \eqref{eq:return count}, the definition of $\Gamma_n$.

  \end{proof}

\section{H\"{o}lder continuity on intervals}\label{se:holder}

This section is devoted to the proof of the first assertion in Theorem~\ref{thm:interval}. We introduce $\bma_r$ in \S~\ref{sse:bma} and show that each individual distribution $F_a$ is H\"{o}lder continuous  (and uniformly in $a$) in \S~\ref{sse:distribution holder} for preparation. Then we complete the proof in \S~\ref{sse:interval proof}.

\subsection{Specifying $\bma_r$}\label{sse:bma}

Recall that $\varphi_n(a)=c_n(a)=f_a^n(c)$ for each $n\ge 0$. By definition, for each $n\ge 2$, $\varphi_n$ is piecewise monotone on $[\sqrt{2},2]$, and restricted to any monotone interval, it is a polynomial of degree $n-1$. Moreover, we have the following basic facts; see, for example, \cite[Lemma~5.1-5.3]{CKY}.

\begin{lem}\label{lem:varphi mono}
Given $n\ge 2$ and an interval $J\subset[\sqrt{2},2]$, $\varphi_n$ is monotone on $J$ iff for each $a$ in the interior of $J$, $c_k(a)\ne c$ for $1\le k<n$. If $\varphi_n$ is monotone on $J$, then we have (for $a\in\partial J$, $\varphi_n'(a)$ is understood as one-sided derivative):
\begin{equation}\label{eq:phase-parameter}
  \sgn \varphi_n'(a)=\sgn (f_a^{n-1})'(c_1), \quad  \forall~a\in J \,,
\end{equation}
and there exists an absolute constant $C>1$ such that
\begin{equation}\label{eq:varphi growth}
 C^{-1} a^n \le |\varphi_n'(a)|\le  C a^n , \quad \forall~a\in J\,.
\end{equation}
%\begin{enumerate}
%  \item $\sgn \varphi_n'=\sgn (f_a^{n-1})'(c_1)$ on $J$.
%  \item $|\varphi_n'|\asymp a^n$ on $J$.
%\end{enumerate}
\end{lem}

%\subsection{Slow return to $I_a^3$}
The following variation of \eqref{eq:varphi growth} is more convenient in application.

\begin{cor}
There exists $C>1$ such that that the following holds. Given $n\ge 2$ and an interval $J\subset[\sqrt{2},2]$, suppose that $\varphi_n$ is monotone on $J$. Then we have:

\begin{equation}\label{eq:distortion}
 C^{-1} a^n |J| \le |\varphi_n(J)| \le  C a^n|J| , \quad \forall~a\in J\,.
\end{equation}
\end{cor}

\begin{proof}
 Denote  $J=[a_1,a_2]$. Then
  \[
 2>|\varphi_n(J)|=\int_J|\varphi_n'(t)|\dif t \ge \delta \int_{a_1}^{a_2} t^{n-1}\dif t =\frac{\delta}{n}(a_2^n-a_1^n),
  \]
where ``$\ge$" is due to  \eqref{eq:varphi growth} and $\delta>0$ is an absolute constant. The line above implies that $(a_2/a_1)^n$ is bounded from above by an absolute constant. Then \eqref{eq:distortion} follows from   \eqref{eq:varphi growth} and mean value theorem.
\end{proof}

The following simple observation is an immediate corollary of \eqref{eq:phase-parameter}.

\begin{prop}\label{prop:window}
There exists an increasing sequence $\sqrt{2}=\bma_2<\bma_3<\cdots $ approaching to $2$, where $\mathbbm{a}_r$ is uniquely determined by:
  \[
  c_2(\mathbbm{a}_r)<c_3(\mathbbm{a}_r) <\cdots < c_{r-1}(\mathbbm{a}_r) < c_r(\mathbbm{a}_r) < c < c_{r+1}(\mathbbm{a}_r) =\hc_2(\mathbbm{a}_r).
  \]
%Equivalently, $\mathbbm{a}_r$ is the unique solution of
%\[
%1+a+\cdots+ a^{r-1}-a^r = a-1,\quad  a\in (1,2).
%\]
Moreover, if $a\in(\bma_r,2)$, then
\[
c_2(a)<c_3(a) <\cdots < c_{r-1}(a) < c_r(a) < c \quad\text{and}\quad c_r(a)<c_{r+1}(a)<\hc_2(a).
\]
\end{prop}

\begin{proof}
  All the statements for $r=2$  hold by definition. By induction, given $r\ge 3$, assume that all the statements hold when the index is strictly less than $r$. Then according to \eqref{eq:phase-parameter} in Lemma~\ref{lem:varphi mono} with $n=r$, $\varphi_{r}$ is strictly decreasing on $[\bma_{r-1},2]$. Then, noting that $\varphi_{r}(2)=-1$, there exists a unique $b\in (\bma_{r-1},2)$ such that $c_r(b)=c$, $c_r(a)>c$ for $a\in [\bma_{r-1},b)$ and $c_r(a)<c$ for $a\in (b,2]$. Applying \eqref{eq:phase-parameter} again with $n=r+1$, we obtain that $\varphi_{r+1}$ is strictly decreasing on $[b,2]$. Noting that $\varphi_{r+1}([b,2])=[-1,1]$, the existence and uniqueness of $\bma_r$ is obtained and the statement about $a\in(\bma_r,2)$ follows from monotonicity of $\varphi_{r+1}$ on $[\bma_r,2]$.
\end{proof}

\begin{rmk}
 The conclusion in Theorem~\ref{thm:interval} (as well as Proposition~\ref{prop:fair meas holder}) implies that $\alpha(r,\bma_r)\le 1$, i.e. $\bma_r^r\ge 2^{r-1}$. In fact, this inequality is strict except for $r=2$. To see this directly, note that $a=\bma_r$ satisfies that
  \[
1+a+\cdots+ a^{r-1}-a^r = a-1 \iff a^r=\frac{(a-1)^2+1}{2-a}.
\]
Denote $\delta:=2-\bma_r\in (0,1)$, then
  \[
  (2-\delta)^r=\bma_r^r=\frac{2}{\delta}-2+\delta.
  \]
By reduction to absurdity, it is easy to see that $r\delta<1$ for $r\ge 3$. It follows that
\[
\bma_r^r=2^r(1-\delta/2)^r>2^r(1-r\delta/2)>2^{r-1}.
\]
\end{rmk}

\subsection{H\"{o}lder continuity of fair distributions}\label{sse:distribution holder}

In this subsection, we aim at proving Proposition~\ref{prop:fair meas holder}, which is based on the following simple observation.
\begin{lem}\label{lem:conformal}
  Given $a\in[\bma_r,2]$ with $r\ge 2$, the following hold:
  \begin{enumerate}
    \item If $\hc_2(a)<x\le c_1$, then $f_a^k(x)<\hc_2(a)$ for $1\le k<r$.
    \item If $J$ is a subinterval of $I_a$ such that $f_a^r$ is injective on $J$, then $\mu_a(f_a^r J)\ge 2^{r-1}\mu_a(J)$.
  \end{enumerate}
\end{lem}

\begin{proof}
  The first assertion follows from Proposition~\ref{prop:window} immediately. The second assertion follows from the first one together with the conformal property \eqref{eq:fair conformal} of $\mu_a$.
\end{proof}

To deduce Proposition~\ref{prop:fair meas holder} from Lemma~\ref{lem:conformal}, we need the following technical lemma as an intermediate step.

\begin{lem}\label{lem:one-step}
  Given $r\ge 2$, there exists $\delta>0$ such that the following holds for each $a\in[\bma_r,2]$. Let $J$  be a subinterval of $I_a$ with $|J|\le\delta$. Suppose that
  \begin{equation}\label{eq:split}
    J=J_1\cup J_2\,, \quad \partial J_1\cap \partial J_2=\{x\}\,,\quad \exists ~ 0\le k<r \ \text{ s.t. }\  f_a^k x=c,
  \end{equation}
where $J_1$, $J_2$ are intervals. Then $k$ is uniquely determined by \eqref{eq:split}, and there exist positive integers $s,t$ with the following properties:
 \begin{enumerate}
   \item $r+1\le t\le 3r+1$ and $rs-(r-1)t\ge 1$.
   \item $f_a^t$ is injective on $J_i$ and $\mu_a(f_a^t J_i)\ge 2^s \mu_a(J_i)$ for $i=1,2$.
 \end{enumerate}

\end{lem}

\begin{proof}
In the proof we will introduce constants $\delta_i>0$ dependent only on $r$, $1\le i\le 5$, and show that the statements hold for $\delta=\min\{\delta_i:1\le i\le 5\}$.  By the definition of $\bma_r$, there exists $\ve=\ve(r)>0$ such that $|c_k(a)-c|\ge \ve$ when $1\le k \le r$ and $a\ge \mathbbm{a}_r$. As a result, there exists $\delta_1>0$ such that the following holds for each $a\ge \mathbbm{a}_r$. If $J$ is an interval  with $|J|\le \delta_1$, then at most one of $J,f_aJ,\cdots, f_a^rJ$ contains $c$. The uniqueness of $k$ in \eqref{eq:split} follows. Now let $J=J_1\cup J_2$ be as in \eqref{eq:split} and $|J|\le \delta_1$. Then $f_a^{r+1}$ is injective on $J_i$.
To proceed, we divide the situation into two cases. Denote $\hI:=(\hc_2(a),c_1]$ for short.

\begin{description}
  \item[Case 1]  $J\cap \hI=\varnothing$. Then $f_a^j J \cap \hI=\varnothing$ for $0\le j\le r$ unless $j=k+1$, provided that $|J|\le \delta_2$ for some $\delta_2>0$ only depends on $r$. The conclusion follows by choosing $(s,t)=(r,r+1)$.
  \item[Case 2]  $J\cap \hI\ne\varnothing$. Then $k=r-1$,  $f_a^j J \cap \hI=\varnothing$ for $1\le j\le r-1$  and $c_{r+1}(a)\in (c,\hc_2(a)]$, provided that $|J|\le \delta_3$ for some $\delta_3>0$ only depends on $r$. Note that $f_a^{r-1}$ is injective on $J$ and $\mu_a(K)\ge 2^{r-2}\mu_a(J)$, where $K:=f_a^{r-1} J$ for short. There are two subcases according as the relative location of $c_{r+1}(a)$ in $(c,\hc_2(a)]$.
      \begin{description}
        \item[Subcase 2-1] $c_{r+1}(a)$ is not close to $\hc_2(a)$. Then $c\notin f_a^j K$ for $1\le j \le r+1$, and $f_a^j K\cap \hI=\varnothing$ for $0\le j\le r+1$ unless $j=1$,  provided that $|J|\le \delta_4$ for some $\delta_4>0$ only depends on $r$. The conclusion follows by choosing $(s,t)=(2r-1,2r+1)$.
        \item[Subcase 2-2] $c_{r+1}(a)$ is  close to $\hc_2(a)$. Then $c\notin f_a^j K$ for $1\le j \le 2r+1$, and for each $x\in K$ the following holds, provided that $|J|\le \delta_5$ for some $\delta_5>0$ only depends on $r$.
            \begin{itemize}
              \item $f_a^j x\notin \hI$ for $0\le j\le 2r+1$ unless $j\in \{1,r+1,2r\}$.
              \item If $f_a^{r+1} x\in \hI$, then $f_a^{2r} x\notin \hI$.
            \end{itemize}
           The conclusion follows by choosing $(s,t)=(3r-2,3r+1)$.
      \end{description}

\end{description}

\end{proof}

Recall the notation of H\"{o}lder exponent $\alpha(r,a):=\frac{(r-1)\log 2}{r\log a}$ defined in \eqref{eq:holder exp}.
\begin{prop}\label{prop:fair meas holder}
  Given $r\ge 2$, there exists $C>0$ such that the following holds for any $\mathbbm{a}_r\le a\le 2$. For any interval $J\subset I_a$, $|\mu_a(J)|\le C |J|^\alpha$, where $\alpha=\alpha(r,a)\in [\frac{1}{2},1]$.
\end{prop}

\begin{proof}
 In the remark following Proposition~\ref{prop:window}, we have shown that $\alpha(r,a)\in [\frac{1}{2},1]$. Fix  $\delta\in (0,1)$ stated in Lemma~\ref{lem:one-step} and we will show the proposition holds for  $C:=\delta^{-1} >1$. Let us argue by induction on length of $J$. To begin with, let $\lambda:=\min_{r\ge 2} 2^{r}(1-2^{-1/r})>1$ and note that $1<\lambda\le \mathbbm{a}_r$. Let $\delta_n:=\delta\cdot \lambda^{-n}$ for $n\ge 0$. According to the choice of $C$, the conclusion holds when $|J|\ge \delta_0$. By induction, suppose that the conclusion holds when $|J|\ge \delta_{n}$ for some $n\ge 0$. Now let $\delta_{n+1}\le|J|<\delta_{n}$. If $f_a^r$ is injective on $J$, then $\mu_a(f_a^rJ)\ge 2^{(r-1)}\mu_a(J)$ and $|f_a^rJ|= a^r|J|>\delta_n$, and hence
\[
\mu_a(J)\le 2^{-(r-1)}\mu_a (f_a^rJ) \le 2^{-(r-1)}\cdot  C \cdot |f_a^rJ|^\alpha  = C \cdot |J|^\alpha.
\]
The induction is completed in this situation. Otherwise, \eqref{eq:split} holds and we are in the position to apply Lemma~\ref{lem:one-step}. Follow the notations in the statement of Lemma~\ref{lem:one-step}, we have:
\[
\mu(J_i) \le 2^{-s} \cdot \mu(f_a^t J_i) \le  2^{-s}\cdot C \cdot\big(\max\{\delta_{n}, |f_a^t J_i|\}\big)^\alpha,\quad i=1,2.
\]
Noting that $|f_a^t J_i|=a^t |J_i|$ and
\[
2^{-s}\cdot a^{t\alpha}=2^{\big(-rs+(r-1)t\big)\big/ r}\le  2^{-1/r},
\]
the estimate above can be written as:
\[
\mu(J_i) \le C\cdot 2^{-1/r} \cdot \big( \max\{ a^{-t}\delta_n, |J_i|\}\big)^\alpha,\quad i=1,2.
\]
We may assume $|J_1|\ge |J_2|$. Then $|J_1|\ge \frac{|J|}{2}\ge a^{-2}\delta_{n+1}>a^{-t}\delta_n$. To complete the induction, it suffices to verify that
\[
   K:=\frac{\sum_{i=1}^2 \big( \max\{  a^{-t}\delta_n, |J_i|\}\big)^\alpha}{|J|^\alpha} = \frac{|J_1|^\alpha + \max\{ a^{-t} \delta_n , |J_2|\}\big)^\alpha}{(|J_1|+|J_2|)^\alpha}< 2^{1/r}.
\]
There are two cases.
\begin{itemize}
  \item If $|J_2|\ge a^{-t}\delta_n$, then  $K= \frac{|J_1|^\alpha +|J_2|^\alpha}{(|J_1|+|J_2|)^\alpha}\le 2^{1-\alpha}<\frac{2}{a^\alpha}=2^{1/r}$.
  \item If $|J_2|< a^{-t}\delta_n$, then $K=\frac{|J_1|^\alpha +(a^{-t}\delta_n)^\alpha}{|J|^\alpha} < 1 +\big(a^{-r-1}\lambda\big)^\alpha  \le 1+2^{-r+1/r}\lambda \le 2^{1/r}$.
\end{itemize}
The induction is completed.

\end{proof}

Let us end this subsection with the following simple fact that might be of independent interest, although we will not use it in this paper.

\begin{cor}
  The function $G$ defined below is continuous:
  \[
  G: \dom(G)=\big\{(a,x) : \sqrt{2}\le a\le 2, x\in I_a \big\}\to [0,1]\,, \quad G(a,x)= F_a(x) \,.
  \]
\end{cor}

\begin{proof}
 From the relation $\Phi_aF_a=F_a$, where $\Phi_a$ is defined by \eqref{eq:transfer} for $f=f_a$, it can be easily seen that $G$ is continuous at $(a,x)$ iff it is continuous at $(a,f_ax)$. It follows that for each $a$, the set $\mclc_a:=\{x\in I_a: G\text{ is continuous at } (a,x) \}$ satisfies that  $f_a^{-1}\mclc_a=\mclc_a$. On the other hand, $c\in \mclc_a$ because of Proposition~\ref{prop:fair meas holder} and continuity of $\mclh$. Therefore, $\mclc_a$ is dense in $I_a$ and contains the critical orbit.

 To complete the proof, given $(a_0,x_0)\in \dom(G)$, let us show that $G$ is continuous at $(a_0,x_0)$, and we may assume that $c_2(a_0)<x_0<c_1$. Given $\ve>0$, let $x_1<x_0<x_2$ be such that $x_1,x_2\in\mclc_{a_0}$ and $|G(a_0,x_1)-G(a_0,x_2)|\le \ve$. Then there exists a closed neighborhood $J$ of $a_0$ in $[\sqrt{2},2]$, such that $|G(a,x_i)-G(a_0,x_i)|\le \ve$ for $a\in J$ and $i=1,2$. Since $G(a,x)=F_a(x)$ is increasing in $x$,
 \[
 \max_{(a,x),(a',x')\in J\times [x_1,x_2]} |G(a,x)-G(a',x')| = \max_{a_1,a_2\in J} |G(a_1,x_1)-G(a_2,x_2)|\le 3\ve,
 \]
which completes the proof.
\end{proof}

\subsection{Lower bound of H\"{o}lder exponents on intervals}\label{sse:interval proof} %{Proof of first assertion in Theorem~\ref{thm:interval}}

%\subsection{Lower bound of H\"{o}lder exponent}

Now we are ready to prove the first assertion in Theorem~\ref{thm:interval}. Given $r\ge 2$,  $b\in(\bma_r,2]$ and  $\bma_r\le a_1<a_2\le b$, to estimate $|\mclh(a_2)-\mclh(a_1)|$, let
\[
n=\max\{m\ge 3: (a_1,a_2) \text{ is contained in a lap of }\varphi_m\}.
\]
By definition, $(a_1,a_2)\subset J$ for some lap $J$ of $\varphi_{n}$, and there exists $a\in (a_1,a_2)$ such that $\tc:=c_{n}(a)\in \{c,\hc_2(a)\}$. It suffices to estimate  $|\mclh(a)-\mclh(a_i)|$. By continuity of $\mclh$, we may assume that $[a_1,a_2]\subset J$. Then by \eqref{eq:fair distribution differ},
\[
- A_{n,J} \cdot\Big( \mclh(a)-\mclh(a_i)\Big)= B_{n,J} \cdot \Big( \big[ F_{a}(\tc)- F_{a_i}(\tc)\big] + \big[ F_{a_i}(c_n(a))-F_{a_i}(c_n(a_i))\big] \Big).
\]
%\[
%- A_{n,J} \cdot\Big( \mclh(a)-\mclh(a_i)\Big)= B_{n,J} \cdot \Big(F_{a}(c_n(a))-  F_{a_i}(c_n(a_i))\Big) = B_{n,J} \cdot \Big( \big(F_{a}(\tc)- F_{a_i}(\tc)\big) + \big( F_{a_i}(\tc)-F_{a_i}(c_n(a_i))\big) \Big).
%\]
It can be rewritten as:
\[
 - \big(A_{n,J}+\eta B_{n,J} \big)\cdot\Big(\mclh(a)-\mclh(a_i)\Big)  =B_{n,J} \cdot  \Big( \big[ F_{a_i}(c_n(a))-F_{a_i}(c_n(a_i))\big] + \ve \Big),
\]
%\[
% - \big(A_{n,J}+\eta B_{n,J} \big)\cdot\Big(\mclh(a)-\mclh(a_i)\Big)  =B_{n,J} \cdot  \big( F_{a_i}(\tc)-F_{a_i}(c_n(a_i))\big) = B_{n,J} \cdot  \big( F_{a_i}(\varphi_n(a))-F_{a_i}(\varphi_n(a_i))\big),
%\]
where $\eta$ and $\ve$ are defined as follows.
\begin{itemize}
  \item If $\tc=c$, then $\eta=1$ and $\ve=0$.
  \item If $\tc=\hc_2(a)$, then $\eta=2$ and $\ve=F_{a_i}(\hc_2(a_i))-F_{a_i}(\hc_2(a))$.
\end{itemize}
By Lemma~\ref{lem:entropy difference}, $|A_{n,J}+\eta B_{n,J}|\ge \frac{1}{2}$ always holds. By Proposition~\ref{prop:fair meas holder} and the definition of $\ve$, there exists a constant $C_1>1$ only dependent on $r$, such that
\[
|\ve|\le C_1\cdot|a-a_i|^{\alpha(r,a_i)}.
\]
It follows that:
\[
 \frac{1}{2}\,|\mclh(a)-\mclh(a_i)| \le |B_{n,J}|\cdot\Big(|F_{a_i}(\varphi_n(a))-F_{a_i}(\varphi_n(a_i))|+ C_1|a-a_i|^{\alpha(r,a_i)} \Big).
\]
Since $J$ is a lap of $\varphi_n$, $\Gamma_{n-1}|_J$ is constant. Then we have:

\[
|B_{n,J}|=2^{-\Gamma_{n-1}|_J}\le  2^{1-(1-1/r)n},
\]
where the ``$=$" is due to Lemma~\ref{lem:entropy difference} and the ``$\le$" follows from $J\subset[\bma_r,2]$ and the first assertion in Lemma~\ref{lem:conformal}.  On the other hand, since $a_i\ge \bma_r$, by Proposition~\ref{prop:fair meas holder}, there exists $C_2>0$ only dependent on $r$, such that
\[
|F_{a_i}(\varphi_n(a))-F_{a_i}(\varphi_n(a_i))|\le C_2\cdot |\varphi_n(a)-\varphi_n(a_i)|^{\alpha(r,a_i)}\,.
\]
Since $a,a_i\in J$ and $\varphi_{n}$ is monotone on $J$, by \eqref{eq:distortion}, there exists an absolute constant $C_3>0$ such that
\[
|\varphi_n(a)-\varphi_n(a_i)|^{\alpha(r,a_i)}  \le C_3 \cdot a_i^{n \cdot\alpha(r,a_i)}\cdot |a-a_i|^{\alpha(r,a_i)}.
\]
Combining all the estimates above together with the relations $a_i^{\alpha(r,a_i)}=2^{1-1/r}$ and $\alpha(r,a_i)\ge \alpha(r,b)$, we obtain that
\[
|\mclh(a)-\mclh(a_i)| \le 4(C_1+C_2C_3)\cdot |a-a_i|^{\alpha(r,b)}.
 \]
The proof of the first assertion in Theorem~\ref{thm:interval} is completed.

\section{Pointwise H\"{o}lder exponents}\label{se:pointwise}

In this section we mainly deal with pointwise H\"{o}lder exponents of $\mclh$ and prove Theorem~\ref{thm:pointwise}; the second assertion in Theorem~\ref{thm:interval} follows as a direct corollary. In \S~\ref{sse:slow rec} we introduce the parameter set $\slowrec$ and show that it is of full measure. In \S~\ref{sse:pointwise proof} we prove the first assertion in  Theorem~\ref{thm:pointwise}. In  \S~\ref{sse:extremal} we prove the second assertion in Theorem~\ref{thm:pointwise} and the second assertion in Theorem~\ref{thm:interval}.

\subsection{Parameter exclusion}\label{sse:slow rec}

\begin{defn}\label{def:slow rec}
 Given $a\in (\sqrt{2},2)$ and $n\ge 3$, denote the lap of $\varphi_n$ containing $a$ by $(a-r_n^{(1)}(a), a+r_n^{(2)}(a))$ if it is well-defined, and denote $r_n^{(1)}(a)=r_n^{(2)}(a)=0$ otherwise. Given $\theta\in (0,1)$, for $i=1,2$, denote
  \[
\msca_i(\theta)=\bigcup_{N=3}^\infty \bigcap_{n=N}^\infty \big\{a\in (\sqrt{2},2): r_n^{(i)}(a)\ge (\theta a^{-1})^n\big\}.
  \]
Moreover, denote
\[
\slowrec=\bigcap_{i=1}^2 \bigcap_{0<\theta<1} \msca_i(\theta).
\]
\end{defn}
By definition, for $i=1,2$,
\[
\bigcup_{n=3}^\infty\{a\in (\sqrt{2},2) : r_n^{(i)}(a)=0\} = \bigcup_{n=3}^\infty\{a\in (\sqrt{2},2) : c_n(a)=c ~\text{or}~ \hc_2(a) \}
\]
is a countable set; $\msca_i(\theta)$ is a Borel set decreasing in $\theta$, so that $\slowrec$ is also Borel.

\begin{prop}\label{prop:slow rec}
  For each $\theta\in (0,1)$, $\msca_i(\theta)$ is of full Lebesgue measure in $[\sqrt{2},2]$, $i=1,2$. As a consequence, $\slowrec$
is of full Lebesgue measure in $[\sqrt{2},2]$.
\end{prop}
The proof of Proposition~\ref{prop:slow rec} is based on Lemma~\ref{lem:mono growth} below. Lemma~\ref{lem:mono growth} might be well known and Proposition~\ref{prop:slow rec} should be obvious to experts. However, we fail to find any explicit reference on either of them, so we provide a self-contained proof here for the reader's convenience.

Before proving Proposition~\ref{prop:slow rec}, let us introduce some notations for preparation. For an interval $I\subset \mbbr$, denote its closure by $\cl I$ and its interior by $\inr I$. Let $\phi:J\to\mbbr$ be a function defined on an interval $J$.  $I\subset J$ is called a {\bf maximal monotone interval} of $\phi$ on $J$, if
\begin{itemize}
  \item $I$ is an open interval and $\phi$ is monotone on $I$;
  \item for any open interval $I'$ with $\cl I\subset I'\subset J$, $\phi$ is not monotone on $I'$.
\end{itemize}

Given $n\ge 3$ and an interval $J\subset [\sqrt{2},2]$, define $\mclm_n(J)$ and $\mcll_n(J)$ as follows.
\[
\mclm_n(J) := \# \{ I\subset J : \text{$I$ is a maximal monotone interval of $\varphi_n$ on $J$} \}.
\]
\[
  \mcll_n(J):=\# \{ L \subset J: \text{$L$ is a lap of $\varphi_n$} \}.
\]
Note that by definition, we always have $\mclm_n(J)\le \mcll_n(J)+2$.
\begin{lem}\label{lem:mono growth}
 Let $J\subset [\sqrt{2},2]$ be an interval. Then we have:
\begin{equation}\label{eq:mono growth}
  \lim_{n\to\infty}\frac{1}{n}\log  \mclm_n(J) =\log \sup J.
\end{equation}
\end{lem}

\begin{proof}
Let us show the ``$\ge$" part first and denote $b=\sup J$. Given $\sqrt{2}\le a<b$, by \eqref{eq:distortion}, there exists an absolute constant $C>1$ such that the following holds: if $I$ is a subinterval of $[a,b]$ and $\varphi_n$ is monotone on $I$, then  $|I|\le C a^{-n}$. It follows that $\mclm_n([a,b])\ge C^{-1}(b-a)a^n$ and hence
  \[
  \liminf_{n\to\infty}\frac{1}{n}\log  \mclm_n(J)\ge \lim_{a\to b} \liminf_{n\to\infty} \frac{1}{n}\log  \mclm_n([a,b])=\log b.
  \]

For the other direction, denote
\[
\mscc_n(J)=\{a\in \inr J: \exists~3\le m<n  \text{ s.t. }  f_a^m(c)=c \},\quad \text{ so that }~ \mclm_n(J) = \# \mscc_n(J) +1.
\]
Given $a\in [\sqrt{2},2]$, denote
\[
M_n(a):=\#\{ I\subset I_a :\text{ $I$ is a maximal monotone interval of $f_a^n$ on $I_a$}\},\quad n\ge 1.
\]
We will make use of two well known facts of $M_n(a)$; see, for example, \cite{MT}. The first one is:
\begin{equation}\label{eq:top entropy}
\lim_{n\to\infty}\frac{1}{n}\log M_n(a) =h_{top}(f_a)=\log a.
\end{equation}
The second is that $M_n(a)$ is increasing in $a$. More precisely, given $(a_1,a_2)\subset [\sqrt{2},2]$, we have:

\begin{itemize}
  \item if $\mscc_n((a_1,a_2))=\varnothing$, then $a\mapsto M_n(a)$ is constant on $(a_1,a_2)$;
  \item if $\mscc_n((a_1,a_2))\ne\varnothing$, then $M_n(a_1)<M_n(a_2)$.
\end{itemize}
It follows that
\[
M_n(\sup J)-M_n(\inf J) \ge \# \mscc_n(J)  = \mclm_n(J)-1.
\]
Combing the line above with \eqref{eq:top entropy}, we have:
\[
 \limsup_{n\to\infty}\frac{1}{n}\log\mclm_n(J) \le \log \sup J.
\]
The proof is completed.
\end{proof}

\begin{cor}\label{cor:lap growth}
    Let  $J\subset [\sqrt{2},2]$ be an interval. Then we have:

\begin{equation}\label{eq:lap growth}
  \lim_{n\to\infty}\frac{1}{n}\log  \mcll_n(J) =\log \sup J.
\end{equation}
\end{cor}

\begin{proof}
  Let us begin with a simple observation. Let $I\subset [\sqrt{2},2]$ be an open interval on which  $\varphi_n$ is monotone, and let $L$ be a lap of $\varphi_n$. If $a\in I$ is an end point of $L$, then $\varphi_{k}(a)=\hc_2(a)=a-1$ for some $3\le k<n$. Let $m\ge 1$ be such that if $k\ge m$, then $|\varphi_k'|>1$ on each monotone interval of $\varphi_k$. Then $\varphi_{k}(a)=\hc_2(a)$ has at most one solution in $I$ for $k\ge m$. It follows that $\mcll_n(I)\le n+C$ for some constant $C>0$ independent of $I$ and $n$. As a result, for any interval $J\subset [\sqrt{2},2]$,
  \[
  \mcll_n(J)\le (n+C) \mclm_n(J),\quad \forall~n\ge 3.
  \]
On the other hand,  $\mcll_n(J)\ge \mclm_n(J)-2$ always holds. Then the conclusion follows from \eqref{eq:mono growth} in Lemma~\ref{lem:mono growth}.

\end{proof}

\begin{proof}[Proof of Proposition~\ref{prop:slow rec}]
 Fix an arbitrary $\theta\in (0,1)$. Then there exist $\sqrt{2}=a_0<a_1<\cdots<a_m=2$ with $a_{k-1}>\theta a_{k}$ for $1\le k\le m$, so
 $[\sqrt{2},2]=\cup_{k=1}^m J_k$, where $J_k=[a_{k-1},a_k]$. It suffices to show that (we use $|\cdot|$ to denote Lebesgue measure on $\mbbr$ below)
 \[
| J_k\setminus  \msca_i(\theta) | =0,\quad 1\le k\le m,~i=1,2.
 \]
By definition,
\[
J_k\setminus  \msca_i(\theta) =\bigcap_{N=3}^\infty \bigcup_{n=N}^\infty \{a\in J_k: r_n^{(i)}(a)< (\theta a^{-1})^n\big\} \subset \bigcap_{N=3}^\infty \bigcup_{n=N}^\infty \{a\in J_k: r_n^{(i)}(a)< (\theta a_{k-1}^{-1})^n\big\}.
\]
By definition, $J_k$ can be covered by at most $\mcll_n(J_k)+2$ laps of $\varphi_n$ together with a finite set. Therefore,
\[
\big|\big\{a\in J_k: r_n^{(i)}(a)< (\theta  a_{k-1}^{-1})^n\big\}\big| \le (\mcll_n(J_k)+2) \cdot (\theta a_{k-1}^{-1})^n.
\]
Combing the line above with \eqref{eq:lap growth} in Corollary~\ref{cor:lap growth} and noting that $\sup J_k \cdot (\theta a_{k-1}^{-1})<1$, we have:
\[
\sum_{n=3}^\infty \big|\big\{a\in J_k: r_n^{(i)}(a)< (\theta a^{-1})^n\big\}\big| <+\infty.
\]
Then $|J_k\setminus  \msca_i(\theta)|=0$ follows from Borel-Cantelli lemma.
\end{proof}

\subsection{Pointwise exponents at parameters in $\slowrec$}\label{sse:pointwise proof}

This subsection is devoted to the proof of the first assertion in  Theorem~\ref{thm:pointwise}. It suffices to prove the following.

\begin{prop}\label{prop:pointwise}
  Given $a\in \slowrec$, the following two equalities hold:
  \[
    \limsup_{b\to a}\frac{\log|\mclh(b)-\mclh(a)|}{\log|b-a|}= \frac{\log 2}{\log a}\cdot \limsup_{n\to\infty} \frac{1}{n}\Gamma_n(a),
  \]
  \[
    \liminf_{b\to a}\frac{\log|\mclh(b)-\mclh(a)|}{\log|b-a|}= \frac{\log 2}{\log a}\cdot \liminf_{n\to\infty} \frac{1}{n}\Gamma_n(a).
  \]
\end{prop}

\begin{rmk}
  From the proof it is easily seen that the assumptions $a\in \cap_{0<\theta<1}\msca_1(\theta)$  and $a\in \cap_{0<\theta<1}\msca_2(\theta)$ correspond to handling $b\to a^-$ and $b\to a^+$ respectively.
\end{rmk}
To prove Proposition~\ref{prop:pointwise}, we adopt the basic strategy in \cite{DM19} to take advantage of monotonicity of $\mclh$, which begins with a simple observation as follows. Let $h:I\to\mbbr$ be a monotone function on an interval $I$.  Let $\Delta_n\searrow 0$ be such that $\lim\limits_{n\to\infty }\frac{\log\Delta_{n+1}}{\log\Delta_n}=1$. Then for $i=1,2$,
\[
\limsup_{\Delta\to 0^+}\frac{\log|h(a+(-1)^i\Delta)-h(a)|}{\log \Delta} = \limsup_{n\to\infty}\frac{\log|h(a+(-1)^i\Delta_n)-h(a)|}{\log \Delta_n},
\]
and
\[
\liminf_{\Delta\to 0^+}\frac{\log|h(a+(-1)^i\Delta)-h(a)|}{\log \Delta} = \liminf_{n\to\infty}\frac{\log|h(a+(-1)^i\Delta_n)-h(a)|}{\log \Delta_n}.
\]
Applying the fact above to $h=\mclh$,  Proposition~\ref{prop:pointwise} is reduced to the following statement.

\begin{clm}
Given $a\in\slowrec$,  there exist two sequences $\Delta_n^{(i)}\searrow 0$, $i=1,2$, satisfying
\[
   \lim_{n\to\infty}\frac{\log\Delta_{n+1}^{(i)}}{\log\Delta_n^{(i)}}=1,
\]
such that for $b_n^{(i)}:=a+(-1)^i \Delta_n^{(i)}$, we have:
  \[
   \limsup_{n\to\infty} \frac{\log |\mclh(b_n^{(i)})-\mclh(a)|}{\log\Delta_n^{(i)}} = \frac{\log 2}{\log a}\cdot \limsup_{n\to\infty} \frac{1}{n}\Gamma_n(a),
  \]
  \[
   \liminf_{n\to\infty} \frac{\log |\mclh(b_n^{(i)})-\mclh(a)|}{\log\Delta_n^{(i)}} = \frac{\log 2}{\log a}\cdot \liminf_{n\to\infty} \frac{1}{n}\Gamma_n(a).
  \]
\end{clm}

 We need the following simple fact in the proof of the claim.

\begin{lem}\label{lem:inverse holder}
  There exists $C>0$ such that for any $a\in [\sqrt{2},2]$ and any interval $J\subset I_a$\:\!, $\mu_a(J)\ge C|J|^4$ holds\footnote{The exponent ``$4$" here is far from being optimal and it can be improved by using a similar argument to the proof of Proposition~\ref{prop:fair meas holder}. }.
\end{lem}
\begin{proof}
Let $J$ be a subinterval of $I_a$. Then the following are evident.
  \begin{itemize}
    \item If $c\notin J$, then $|f_aJ|=a|J|\ge \sqrt{2} |J|$ and $|\mu_a(f_aJ)|\le 2\mu_a(J)$, so that $\frac{\mu_a(J)}{|J|^4}\ge 2\cdot\frac{\mu_a(f_aJ)}{|f_aJ|^4}$.
    \item If $c\in J$, $f_a(J)\subset [\hc_2(a), c_1]$ and  $f_a^2(J)\subset [c_2(a), c]$, then $\mu_a(f_a^3J)= 4 \mu_a (J)$ and $|f_a^3J|\ge \frac{a^3}{2}|J|\ge \sqrt{2}|J|$, so that $\frac{\mu_a(J)}{|J|^4}\ge \frac{\mu_a(f_aJ)}{|f_aJ|^4}$.
    \item If  either $f_a(J)\supset [\hc_2(a), c_1]$ or  $f_a^2(J)\supset [c_2(a), c]$, then  $\frac{\mu_a(J)}{|J|^4}\ge C$ for some absolute constant $C>0$.
  \end{itemize}
By induction on the length of $J$, the conclusion follows easily from the facts above.
\end{proof}

\begin{proof}[Proof of claim]

Fix $a\in\slowrec$ below. For $r_t^{(i)}=r_t^{(i)}(a)$ introduced in Definition~\ref{def:slow rec}, denote
\[
a_t^{(i)}:=a+(-1)^i r_t^{(i)},\quad J_t^{(1)}:=(a-r_t^{(1)},a)=(a_t^{(1)},a), \quad J_t^{(2)}:=(a,a+r_t^{(2)})=(a,a_t^{(2)}).
\]
Given $t\ge 3$, let
\[
s_i(t)=\min\{s>t: J_t^{(i)} \text{ is not contained in a lap of } \varphi_s \}.
\]
By definition,
\[
\varphi_{s_i(t)-1}\big(a_{s_i(t)}^{(i)}\big)=c  \ \text{ or }\ \hc_2(a_{s_i(t)}^{(i)}).
\]
From $a\in \msca_i(\theta)$ we know that $r_{t}^{(i)}\ge (\theta a^{-1})^t$ holds for large $t$; by \eqref{eq:distortion} we have $r_{s_i(t)}^{(i)}\le C a^{-s_i(t)}$ for any $t\ge 3$, where $C>0$ is an absolute constant. Combing these facts with $r_{t}^{(i)}\ge r_{s_i(t)}^{(i)}$ yields that
\[
\frac{\log a}{\log a-\log\theta} \cdot \limsup_{t\to\infty}\frac{s_i(t)}{t} \le \limsup_{t\to\infty}\frac{\log r_{s_i(t)}^{(i)}}{\log r_{t}^{(i)}} \le  1.
\]
Letting $\theta\to 1^-$, we obtain that
\begin{equation}\label{eq:ratio time}
  \lim_{t\to\infty}\frac{s_i(t)}{t}= \lim_{t\to\infty}\frac{\log r_{s_i(t)}^{(i)}}{\log r_{t}^{(i)}}=1,\quad i=1,2.
\end{equation}

Let $t_0^{(i)}:=3$ and define natural numbers $t_n^{(i)}\nearrow \infty$ inductively:  once $t_n^{(i)}$ is defined for some $n\ge 0$, let  $t_{n+1}^{(i)}:=s_i(t_n^{(i)})$. Now let $\Delta_n^{(i)}:=r_{t_n^{(i)}}$ for $n\ge 0$. By definition and \eqref{eq:ratio time},  $\lim\limits_{n\to\infty}\frac{\log\Delta_{n+1}^{(i)}}{\log\Delta_n^{(i)}}=1$ holds.

To complete the proof,  let $L_i$ denote the lap of $\varphi_{t_n^{(i)}-1}$ containing $a$. Then for $b_n^{(i)}=a+(-1)^i \Delta_n^{(i)}$,
\[
\tc_i :=c_{t_n^{(i)}-1}(b_n^{(i)}) = c ~~\text{or}~~ \hc_2( b_n^{(i)}).
\]
By \eqref{eq:fair distribution differ}, for $A_i:=A_{t_n^{(i)}-1,L_i}$ and $B_i:=B_{t_n^{(i)}-1,L_i}$, we have:

\[
 - A_i\cdot\Big(\mclh(a)-\mclh(b_n^{(i)})\Big)=B_i \cdot\Big( \big[F_a(c_{t_n^{(i)}-1}(a)) - F_{a}(c_{t_n^{(i)}-1}(b_n^{(i)}) )\big]  + \big[ F_{a}(\tc_i) - F_{b_n^{(i)}}(\tc_i)\big] \Big).
\]

Repeating we did in \S~\ref{sse:interval proof}, the line above can be rewritten as:
\[
- (A_i+\eta_i B_i) \Big(\mclh(a)-\mclh(b_n^{(i)})\Big) = B_i \cdot \Big(  \big[F_a(c_{t_n^{(i)}-1}(a)) - F_{a}(c_{t_n^{(i)}-1}(b_n^{(i)}) )\big] + \ve_i \Big),
\]
where $\eta_i$ and $\ve_i$ are defined as follows.
\begin{itemize}
  \item If $\tc_i=c$, then $\eta_i=1$ and $\ve_i=0$.
  \item If $\tc_i=\hc_2(b_n^{(i)})$, then $\eta_i=2$ and $\ve_i=F_{a}(\hc_2(b_n^{(i)}))-F_{a}(\hc_2(a))$.
\end{itemize}
According to Lemma~\ref{lem:entropy difference}, we have:
\[
\frac{1}{2}< |A_i+\eta_i B_i|< 3  \quad\text{and}\quad \log |B_i|=-\,\Gamma_{t_n^{(i)}-1}(a)\cdot \log 2.
\]
By Proposition~\ref{prop:fair meas holder} and the definition of $\ve_i$, there exists a constant $C>1$ only dependent on $r$, such that
\[
|\ve_i|\le C\cdot|b_n^{(i)}-a|^{\alpha(r,a)}\le C\cdot\big(\Delta_n^{(i)}\big)^{\frac{1}{2}}.
\]
Also note that
\[
  \lim_{n\to\infty}\frac{\log \Delta_n^{(i)}}{t_n^{(i)}} =-\log a.
\]
Combining all the four lines of displayed equations above, to complete the proof, it remains to show that
\begin{equation}\label{eq:tempered}
  \lim_{n\to\infty }\frac{\log |F_a(\varphi_{t_n^{(i)}-1}(a)) - F_a(\varphi_{t_n^{(i)}-1}(b_n^{(i)}))| }{ t_n^{(i)}} =0.
\end{equation}
To verify \eqref{eq:tempered}, firstly, note that
  \[
  C_1\cdot|\varphi_{t_n^{(i)}-1}(a) - \varphi_{t_n^{(i)}-1}(b_n^{(i)})|^4 \le |F_a(\varphi_{t_n^{(i)}-1}(a)) - F_{a}(\varphi_{t_n^{(i)}-1}(b_n^{(i)}))|<1,
  \]
where the ``$\le$" is due to Lemma~\ref{lem:inverse holder}  and $C_1>0$ is an absolute constant. Secondly, by \eqref{eq:distortion}, there exists another absolute constant $C_2>1$ such that
  \[
  C_2^{-1} a^{t_n^{(i)}} \Delta_n^{(i)} \le |\varphi_{t_n^{(i)}-1}(a) - \varphi_{t_n^{(i)}-1}(b_n^{(i)})| \le C_2 a^{t_n^{(i)}} \Delta_n^{(i)}.
  \]
Then   \eqref{eq:tempered} follows and the proof is completed.
\end{proof}

\subsection{Extreme values of $\gamma$}\label{sse:extremal}

This subsection is devoted to the proof of the second assertion in  Theorem~\ref{thm:pointwise}, as well as the second assertion in  Theorem~\ref{thm:interval}.

By Proposition~\ref{prop:window}, the following lemma is self-evident.

\begin{lem}\label{lem:periodic orbit}
Given $r\ge 2$ and $a\in (\mathbbm{a}_r,\mathbbm{a}_{r+1})$,
\[
q_r(a):=\frac{1+a+\cdots+a^{r-1}}{1+a^r}\in (\hc_2(a),c_1)
\]
is an orientation reversing fixed point of $f_a^r$. In other words, $q=q_r(a)$ is determined by
    \[
    f_a(q)<f_a^2(q)<\cdots < f_a^{r-1}(q)< c <\hc_2(a) <f_a^r(q)=q< c_1.
    \]

\end{lem}

On the other hand, for $r=1$ and $a\in (\sqrt{2},2)$, denote
\[
  q_1(a):=\frac{1}{1+a}\in (c,\hc_2(a))\,,
 \]
i.e. $q_1(a)$ is the orientation-reversing fixed point of $f_a$.

The following two elementary lemmas are well known. See, for example, \cite[Lemma~5.5]{CKY} and \cite[Lemma~7.3]{CKY}  respectively.
\begin{lem}\label{lem:periodic cirt dense}
  The following set
\[
 \mscp:=\{a\in (\sqrt{2},2): \exists~n\ge 1 \text{ s.t. } c_n(a)=c\}
\]
is dense in $(\sqrt{2},2)$.
%  \begin{equation}\label{eq:pc parameter}
%    \mscp:=\{a\in[\sqrt{2},2]: \exists~n\ge 1 \text{ s.t. } c_n(a)=c\}.
%  \end{equation}
\end{lem}

\begin{lem}\label{lem:prefix dense}
 The following set
\[
    \mscq_1:=\{a\in(\sqrt{2},2): \exists~n\ge 1 \text{ s.t. } c_n(a)=q_1(a)\}
\]
is dense in $(\sqrt{2},2)$.
%  \begin{equation}\label{eq:pc parameter}
%    \mscq_1:=\{a\in[\sqrt{2},2]: \exists~n\ge 1 \text{ s.t. } c_n(a)=q_1(a)\}.
%  \end{equation}
\end{lem}

\begin{cor}\label{cor:preperiodic dense}
   Given $r\ge 2$, the following set
  \[
  \mscq_r:=\big\{a\in (\mathbbm{a}_r,\mathbbm{a}_{r+1}) :\exists~n\ge 1 \text{ s.t. } c_n(a)=q_r(a) \big\}
  \]
is dense in $(\mathbbm{a}_r,\mathbbm{a}_{r+1})$.
\end{cor}

\begin{proof}
Let $J\subset (\mathbbm{a}_r,\mathbbm{a}_{r+1})$ be an open interval. It suffices to find $a\in J$ and $n\ge 1$ such that  $c_n(a)=q_r(a)$. By Lemma~\ref{lem:periodic cirt dense}, there exist $n_1\ge 1$ and $a_1\in J$ such that $\varphi_{n_1}(a_1)=c$. By Lemma~\ref{lem:prefix dense}, there exist
$n_2\ge 1$ and $a_2\in J$ such that $\varphi_{n_2}(a_2)=q_1(a_2)$ and $a_2>a_1$. Let $k\ge 1$ be such that $n:=kn_1+1\ge n_2$. Then $\varphi_n(a_1)=c_1$ and $\varphi_n(a_2)=q_1(a_2)<\hc_2(a_2)$, and hence the continuous function $a\mapsto \varphi_n(a)-q_r(a)$ has opposite sign at two end points of $[a_1,a_2]$. The conclusion follows from intermediate value theorem.
 \end{proof}

The second assertion in  Theorem~\ref{thm:pointwise} follows immediately from Lemma~\ref{lem:prefix dense} and Corollary~\ref{cor:preperiodic dense} and the obvious facts below.

\begin{itemize}
  \item For $a\in(\sqrt{2},2)$, if $c_m(a)\notin\{ c, \hc_2(a)\}$ is a periodic point of $f_a$ for some $m$, then the two sequences $r_n^{(i)}(a)\cdot a^n$, $i=1,2$, introduced in Definition~\ref{def:slow rec} are bounded from below. In particular, $\cup_{r=1}^\infty \mscq_r\subset \slowrec$.
  \item Given $r\ge 2$ and  $a\in(\bma_{r},2]$, by Proposition~\ref{prop:window},  $\liminf\limits_{n\to\infty}\frac{1}{n}\Gamma_n(a)\ge\frac{r-1}{r}$.
  \item If $a\in \mscq_1$, then $\gamma(a)=1$.
  \item For each $r\ge 2$, if $a\in\mscq_r$, then $\gamma(a)=\frac{r-1}{r}$.
\end{itemize}
As a corollary, given $r\ge 2$ and $a\in\mscq_r$, $\beta(a)=\frac{(r-1)\log 2}{r\log a}=\alpha(r,a)$. Since $\mscq_r$ is dense in $(\mathbbm{a}_r,\mathbbm{a}_{r+1})$, the second assertion in Theorem~\ref{thm:interval} follows.

\subsection*{Acknowledgments}
We would like to thank Yiwei Zhang for attracting our attention to this topic. We also would like to thank Gang Liao for helpful comments.

\bibliographystyle{abbrv}             %    other options: alpha, plain, abbrv ...

\bibliography{refer}

%\fi

\end{document}